\documentclass[11pt]{article}

\usepackage{pgf,tikz}
\usepackage{mathrsfs}
\usetikzlibrary{arrows}
\usepackage{lmodern}			
\usepackage[T1]{fontenc}		
\usepackage[utf8]{inputenc}		
\usepackage{lastpage}			
\usepackage{indentfirst}

\usepackage{graphics}
\usepackage{subfigure}
\usepackage{graphicx}
\usepackage{epsfig}
\usepackage[centertags]{amsmath}
\usepackage{graphicx,indentfirst,amsmath,amsfonts,amssymb,amsthm,newlfont}
\usepackage{longtable}
\usepackage{cite}

\theoremstyle{plain} 
\newtheorem{theorem}{Theorem}
\newtheorem{corollary}{Corollary}[theorem]
\theoremstyle{definition}
\newtheorem{definition}{Definition}[section]

\begin{document}

\pagenumbering{arabic}

\title{Conservation laws for the Kudryashov-Sinelshchikov equation of second order}

\author{Yuri Dimitrov Bozhkov  \\
Instituto de Matemática, Estatística e Computação Científica--IMECC,\\
 Universidade Estadual de Campinas -- UNICAMP \\ 
Rua S\'ergio Buarque de Holanda, 651, $13083$-$859$ - Campinas - SP, Brasil, \\
bozhkov@ime.unicamp.br
\and Stylianos Dimas  \\
Departamento de Matem\'atica, \\ Instituto Tecnol\'ogico da Aeron\'autica - ITA, \\
Pra\c ca Marechal Eduardo Gomes, 50,
Vila das Ac\'acias, \\ 
$12228-900$ - S\~ao Jose dos Campos - SP, Brasil, \\
sdimas@ita.br
\and Oscar Mario Londoño Duque \\
Instituto de Matemática, Estatística e Computação Científica--IMECC,\\
 Universidade Estadual de Campinas -- UNICAMP \\ 
Rua S\'ergio Buarque de Holanda, 651, $13083$-$859$ - Campinas - SP, Brasil, \\
olondon2@hotmail.com,	ra154278@ime.unicamp.br}

\date{19 April 2019}
\maketitle

\noindent{{\bf Abstract:} \rm We establish conservation laws for the second order Kudryashov– Sinelshchikov equation, which models pressure waves in liquid with bubbles. For this purpose we use the method of Nail Ibragimov based on the notion of nonlinear self-adjointness.} \\


\section*{This paper is dedicated to Professor Nail Ibragimov in Memoriam}

\section{Introduction} 	
In \cite{CA.1}, Kudryashov and Sinelshchikov studied nonlinear pressure waves in a mixture of a liquid
and gas bubbles. Among other results and considerations, they obtained the following second order partial differential equation:
\begin{equation}\label{eqq1}
   u_t+ u u_x-(1+u)u_{xx}-u_{x}^2 =0.
\end{equation}
Here $t$ and $x$ are the independent variables and $u=u(t,x)$ denotes the dimensionless dependent variable. We shall call the Eq. $(\ref{eqq1})$ the Kudryashov-Sinelshchikov equation. 

The purpose of this paper is to establish nontrivial conservation laws by using the N. Ibragimov’s conservation theorem \cite{i2,i4,i6,i7}, which is done in section 3. For this aim we first study in section 2 the nonlinear self-adjointness of the more general class of partial differential equations
\begin{equation}\label{eq1}
u_t-f(u)u_x-g(u)u_{xx}+h(u)u_{x}^2 =0,
\end{equation}
where the functions $f,g,h$ satisfy the conditions
\begin{equation}\label{az}
f\neq 0,\; g\neq 0,\; g'\neq 0,
\end{equation}
and $'$ denotes the derivative $d/du$. 

We suppose that the reader is familiar with the basic concepts of Lie point symmetries of differential equations. The fundamental monographs in this area are \cite{ba,bcha,bk,i,ibr,ol,ov} .

\section{Nonlinear Self-Adjointness}

We first briefly present the main definitions in the N. Ibragimov's approach to nonlinear self-adjointness of differential equations adopted to our specific case. For further details the interested reader is directed to \cite{i7,i6,i4,i2,gan}.

Consider an $s-$th order evolution equation
\begin{equation}\label{ad1}
\mathfrak{F}(t,x,u_{(1)},u_{(2)},\cdots,u_{(s)})=0
\end{equation}
with  independent variables $t,x$ and a dependent variable $u$, where 
$u_{(1)},u_{(2)}$,.\newline ..$u_{(s)}$ denote the collection of $1,2,...,s-$th order partial derivatives of $u$.

\begin{definition}
 Let $\mathfrak{F}$ be a differential function  and $\nu = \nu(t,x)-$the new dependent variable, known as the adjoint variable or nonlocal variable \cite{i7}. The formal Lagrangian for  $\mathfrak{F}=0$ is the differential function defined by
   \begin{equation}\label{ad2}
   \mathfrak{L}:=\nu \mathfrak{F}.
   \end{equation}
\end{definition}

\begin{definition}
Let $\mathfrak{F}$ be a differential function and for the differential equation (\ref{ad1}), denoted by $\mathfrak{F}[u] = 0$, we define the adjoint differential function to $\mathfrak{F}$ by
\begin{equation}\label{ad3}
\mathfrak{F}^{*}:=\frac{\delta \mathfrak{L} }{\delta u}
\end{equation}
and the adjoint differential equation by
\begin{equation}\label{ad4}
\mathfrak{F}^{*}[u,\nu] = 0,
\end{equation}
where the Euler operator
\begin{equation}\label{ad5}
\frac{\delta}{\delta u}=\frac{\partial}{\partial u}+\sum_{m=1}^{\infty}D_{x_{i_{1}}}\cdots D_{x_{i_{m}}}\frac{\partial}{\partial {u_{x_{i_{1}}x_{i_{2}}\cdots x_{i_{m}}}}}
\end{equation}
and $D_{x_{i}}$is the total derivative operator with respect to $x_i$ defined by
\[D_{x_{i}}=\partial_{x_{i}}+u_{x_{i}}\partial_{u}+u_{x_{i}x_{j}}\partial_{u_{x_{j}}}+\cdots +u_{x_{i}x_{i_{1}}x_{i_{2}}\cdots x_{i_{n}}}\partial_{u_{x_{i_{1}}x_{i_{2}}\cdots x_{i_{n}}}}... \]
\end{definition}
\begin{definition}
The differential equation (\ref{ad1}) is said to be nonlinearly self-adjoint if there exists a substitution 
\begin{equation}\label{ad6}
\nu=\phi(t,x,u)\neq 0
\end{equation}
such that 
\begin{equation}\label{ad7}
\mathfrak{F}^{*}\mid_{\nu=\phi(t,x,u)}=\lambda \mathfrak{F}
\end{equation}
for some undetermined coefficient $\lambda =\lambda (t,x,u,\cdots ).$ If $\nu=\phi(u)$ in (\ref{ad6}) and (\ref{ad7}), the equation
(\ref{ad1}) is called quasi self-adjoint. If $\nu=u$, we say that the equation (\ref{ad1}) is strictly self-adjoint.
\end{definition}
Now we shall obtain the adjoint equation to the eq. (\ref{eq1}). For this purpose we write (\ref{eq1}) in the form (\ref{ad1}), where
\begin{equation}\label{ad8}
\mathfrak{F}:=u_t-f(u)u_x-g(u)u_{xx}+h(u)u_{x}^2 =0.
\end{equation}
Then the corresponding formal Lagrangian (\ref{ad2}) is given by 
\begin{equation}\label{ad9}
   \mathfrak{L}:=\nu (u_t-f(u)u_x-g(u)u_{xx}+h(u)u_{x}^2 )=0
   \end{equation}
   and the Euler operator (\ref{ad5}) assumes the following form:
   \begin{equation}\label{ad10}
\frac{\delta \mathfrak{L}}{\delta u}=\frac{\partial \mathfrak{L}}{\partial u}-D_{t}\frac{\partial \mathfrak{L}}{\partial {u_{t}}}-D_{x}\frac{\partial \mathfrak{L}}{\partial {u_{x}}}
+D_{x}^{2}\frac{\partial \mathfrak{L}}{\partial {u_{xx}}}.
\end{equation}
We calculate explicitly the Euler operator (\ref{ad10}) applied to $\mathfrak{L}$ determined by (\ref{ad9}). In this way we
obtain the adjoint equation (\ref{ad4}) to (\ref{eq1}):
\begin{eqnarray}\label{ad11}
\mathfrak{F}^{*}&=&-\nu _t+f(u)\nu _x-g(u)\nu_{xx}-2\nu g'(u)u_{xx}-2\nu h(u)u_{xx}\nonumber\\
&-&\nu h'(u)u_{x}^2-\nu g''(u)u_{x}^2-2g'(u)\nu_{x}u_{x}-2h(u)u_{x}\nu_{x}=0.
\end{eqnarray}

The main result in this section can be stated as follows.

\begin{theorem} The generalized  Kudryashov-Sinelshchikov equation $(\ref{eq1})$ is  nonlinearly self-adjoint. 

If $f(u) = a g(u) +b$, $a,b-$constants, and $f(u)$ itself is not a constant, then the substitution is given by
\begin{equation}
\phi(t,x,u)=(c_1 \exp (ax + abt) + c_2)\exp \left ({\displaystyle{-\int{\frac{g'(u)+h(u)}{g(u)}du}}}\right ).
\end{equation}

If $f(u) = b-$constant, then the corresponding substitution is given by
\begin{equation}
\phi(t,x,u)=(c_3 (x + bt) + c_4)\exp \left ({\displaystyle{-\int{\frac{g'(u)+h(u)}{g(u)}du}}}\right ).
\end{equation}

In the rest of the cases, the substitution is 
\begin{align}\label{uu1}
\phi(u)=c_5\exp \left ({\displaystyle{-\int{\frac{g'(u)+h(u)}{g(u)}du}}}\right ).& 
\end{align}

The constants $c_1,c_2,c_3,c_4,c_5$ are such that $(c_1,c_2)\neq (0,0)$, $(c_3,c_4)\neq (0,0)$ and $c_5\neq 0$.
\end{theorem}
\begin{proof}
 Substituting in (\ref{ad11}), and then in (\ref{ad7}), $\nu = \phi (t,x,u)$ and its respective partial derivatives, and comparing the corresponding coefficients we get four equations:
\begin{equation}\label{sad16}
 g(u)\phi_{u}+(h(u)+g'(u))\phi =0,
 \end{equation}
 \begin{equation}\label{sad16.3}
 g(u)\phi_{uu}+(h(u)+2g'(u))\phi_{u}+(h'(u)+g''(u))\phi=0.
 \end{equation}
\begin{equation}\label{vv1}
 g(u)\phi_{xu}+(h(u)+g'(u))\phi_x =0,
 \end{equation}
\begin{equation}\label{us}
 \phi_{t}-f(u)\phi_x+g(u))\phi_{xx} =0
 \end{equation}
We observe that Eq.(\ref{sad16.3}) is obtaned from Eq.(\ref{sad16}) by differentiation with respect to $u$ as well as that Eq.(\ref{vv1}) is obtaned from Eq.(\ref{sad16}) by differentiation with respect to $x$. Therefore we have to study only Eqs. (\ref{sad16}) and (\ref{us}).

Solving for $\phi$ in (\ref{sad16}) we obtain 
\begin{equation}\label{vv3}
 \phi(t,x,u)=M(t,x)\exp \left ({\displaystyle{-\int{\frac{g'(u)+h(u)}{g(u)}du}}}\right )
 \end{equation}
for certain function $M(t,x)$. Substituting (\ref{vv3}) into (\ref{us}) we get that  
\begin{equation}\label{vv2}
M_{t}-f(u)M_x+g(u)M_{xx} =0.
 \end{equation}

Consider now the Wronskian \[W=W(1,-f,g)=f''g'-f'g''.\]

1.) If $W\neq 0$ then $M_t=M_x=0$ from (\ref{vv2}). Hence $M=$const., and in this case the equation is quasi self-adjoint.

2.) Let $W=0$.  Since $g'\neq 0$ (see (\ref{az})) we divide by $g'^2$ and obtain
\[ \left ( \frac{f'(u)}{g'(u)}\right )'=0.\]
Thus $f'=a g'$ and $f(u) = a g(u) +b$, with $a,b-$constants. 

Let $a\neq 0$. We note that this condition is equivalent to the fact that the function $f(u)$ is not a constant.

We substitute this form of $f$ into (\ref{vv2}):
\begin{equation} \label{hh3}
M_t - b M_x + g(u)(M_{xx}-a M_x) =0.
\end{equation}
Differentiating (\ref{hh3}) with respect to $u$ we obtain that $(M_{xx}-a M_x)g'(u)=0$. Hence  
\begin{equation} \label{vv6} M_{xx}-a M_x=0\end{equation}
 since $g'\neq 0$. Further we substitute the general solution 
\[ M(t,x) = c_1(t) \exp (ax) +c_2(t)\]
of the last equation into $M_t-b M_x=0$ (see (\ref{hh3})), and obtain that 
\[ M(t,x) = c_1 \exp (ax + abt) + c_2,\]
where $c_1,c_2$ are arbitrary constants and $a\neq 0$. 

If $a=0$ then $f(u)=b-$constant and $M_{xx}=0$ from (\ref{vv6}). Therefore $M(t,x)= c_3(t) x +c_4(t)$. Substituting this into $M_t-b M_x=0$ we finally obtain that $M(t,x)= c_3(x+bt) +c_4$, where $c_3,c_4$ are arbitrary constants. \end{proof}

{\bf Observation.} The equations (\ref{sad16})--(\ref{us}) can be directly obtained from \cite{bs}.

This theorem has some immedate consequences.
\begin{corollary}\label{teor4.1} The generalized  Kudryashov-Sinelshchikov equation $(\ref{eq1})$ is  quasi self-adjoint with the substitution
\begin{align}\label{uu}
\phi(u)=M\exp \left ({\displaystyle{-\int{\frac{g'(u)+h(u)}{g(u)}du}}}\right )& 
\end{align}
where $M\neq 0$ is a constant.
\end{corollary}
\begin{corollary}\label{teor4}
The generalized  Kudryashov-Sinelshchikov equation $(\ref{eq1})$ is  strictly self-adjoint if and only
 if
 \begin{align}\label{sad1}
 &h(u)=-\frac{g(u)}{u}-g'(u).&  
\end{align}
\end{corollary}
\begin{corollary} The Kudryashov-Sinelshchikov equation $(\ref{eqq1})$ is  nonlinearly self-adjoint with the substitution
 \begin{align}\label{sad13}
\phi(u)=c_1 \exp (-t-x) + c_2& 
\end{align}
with $(c_1,c_2)\neq (0,0)-$constants. It is not strictly self-adjoint equation.
\end{corollary}

The results in this section are compatible with the considerations of N. Ibragimov, \cite{i7}, pp. 20--21, on the self-adjointnes of the nonlinear heat equation.

\section{Conservation laws} 

Following the classical S. Lie algorithm it is easy to obtain that the Lie point symmetries of equation (\ref{eqq1}) are determined by the vector fields
\begin{equation}\label{lps} 
X_1=\partial_t, \;\; X_2=\partial_x,\;\; X_3=t\partial_t-t\partial_x-(1+u)\partial_u.
\end{equation}

In this section we shall establish some conservation laws for the Kudrya-shov-Sinelshchikov equation (\ref{eqq1}) using the conservation theorem of N. Ibra-gimov in \cite{i7}.

To begin with, we observe that Eq. (\ref{eqq1}) itself is a conservation laws since it can be written in the form 
\begin{equation}\label{n1}
D_t u + D_x [u^2/2 - (1+u)u_x]=0.
\end{equation}

Since the Eq. (\ref{eqq1}) is of second order, the formal Lagrangian contains derivatives up to order  two. Thus, the general formula in \cite{i7} for the components of the conserved vector is reduced to 
\begin{eqnarray}\label{lc4}
C^{i}&=&W\left [  \frac{\partial \mathfrak{L}}{\partial_{u_i}}-D_{j}(\frac{\partial \mathfrak{L}}{\partial_{u_{ij}}})\right ]+D_{j}[W]\left [\frac{\partial \mathfrak{L}}{\partial_{u_{ij}}}  \right ],
\end{eqnarray}
where 
\[W=\eta-{\xi}^j u_j,\] 
$i,j=1,2$, the formal Lagrangian 
\[ \mathfrak{L}=\nu (u_t+ u u_x-(1+u)u_{xx}-u_{x}^2)\] 
and ${\xi}^1,{\xi}^2,\eta $ are the infinitesimals of a Lie point symmetry admitted by 
Eq. (\ref{eqq1}), given in (\ref{lps}). 

If $c_1=0$ in the substitution (\ref{sad13}), then the resulting conservation laws are trivial. Therefore we can set $c_1=1$ and $c_2=0$ without loss of generality. Hence
\[ \nu= \exp(-t-x).\]
Using (\ref{eqq1}) and (\ref{lc4}) we obtain that the symmetry $X_2$ determines the following conserved vector $C=(C^{1},C^{2})$, where 
\[ C^1=-u_x \exp(-t-x),\;\;\;C^2= (u_t+(1+u)u_x)\exp(-t-x).\]

Similarly, for $X_1$ we get 
\[ C^1=-u_t \exp(-t-x),\;\;\;C^2= [u_t(1+u_x)+(1+u)u_{tx})\exp(-t-x).\]
Further, we apply the simplifying procedure of transfering terms of form $D_x(...)$ from $C^1$ to $C^2$, described on pp. 50--51 of \cite{i7}, in order to eliminate the second partial derivatives of $u$. After some work, the components of the conserved vector $C=(C^{1},C^{2})$ become 
\[ C^1=[u^2/2-(1+u)u_{x}] \exp(-t-x),\] 
\[C^2= [-u^2/2 + (1+u)u_t+(1+u)u_{x}]\exp(-t-x).\]

Analogously, for the symmetry $X_3$ we obtain conserved vector $C=(C^{1},C^{2})$ with
\[ C^1=[-(1+u)+t u^2/2-t u u_{x}] \exp(-t-x),\] 
\[C^2= [(1+u) +(1+u)u_x +u^2/2 + t u u_t-t u^2/2]
\exp(-t-x).\]

\end{document}